\documentclass[reqno]{amsart}
\usepackage{amsmath}
\usepackage{amssymb}
\usepackage{xypic}
\usepackage{mathrsfs}
\usepackage{hyperref}
\usepackage{graphicx}
\usepackage{perpage}
\MakePerPage{footnote}
\usepackage{kuvio}

\date{\today}
          
 \usepackage[usenames,dvipsnames]{pstricks}
\usepackage[shell]{pdftricks}
 \begin{psinputs}
 \usepackage{pst-grad} % For gradients
 \usepackage{pst-plot} % For axes
 \usepackage{epsfig}
 \end{psinputs}

\title{A characterization of Whitney $a$-regular complex analytic stratifications}

\author{Saurabh Trivedi \vspace{-0.8cm}} 
\address{LATP (UMR 7353), Centre de Math\'ematiques et Informatique, Aix-Marseille Universit\'e, 39 rue Joliot-Curie, 13453 Marseille Cedex 13, France.}

\email{saurabh@cmi.univ-mrs.fr}

\begin{document}
\maketitle
\newtheorem{thm}{Theorem}[section]
\newtheorem{lem}[thm]{Lemma}

\theoremstyle{theorem}
\newtheorem{prop}[thm]{Proposition}
\newtheorem{defn}[thm]{Definition}
\newtheorem{cor}[thm]{Corollary}
\newtheorem{example}[thm]{Example}
\newtheorem{xca}[thm]{Exercise}

\theoremstyle{remark}
\newtheorem{rem}[thm]{Remark}

\numberwithin{equation}{section}

\newenvironment{pf}{\noindent \textbf{Proof:}}{\hfill $\square$\\}
\newcommand{\bb}{\mathbb}
\newcommand{\al}{\mathcal}
\newcommand{\ak}{\mathfrak}
\newcommand{\fs}{\mathscr}

\renewcommand{\thefootnote}{\fnsymbol{footnote}}

\begin{abstract} 
We prove that the openness of the set of maps, between a Stein manifold and an Oka manifold, transverse to a stratification of a complex analytic subvariety in the target implies that the stratification is Whitney $a$-regular. Our result can be seen as a complex version of Trotman's theorem. 
\end{abstract}

\parskip .1cm

\section{Introduction}
A very famous result of Whitney \cite{Whitney} states that any complex analytic variety\footnote{For us a complex analytic variety is always closed.} in a complex manifold can be stratified into complex analytic strata. And, it can be stratified in such a way that the strata satisfy certain regularity conditions, like $a$ and $b$ regularity\footnote{See Mather \cite{Mather3} for definitions of stratifications and regularity conditions $a$ and $b$.} . Similar results were later found to hold for stratifications with semialgebraic, subanalytic strata and more generally with definable stratifications in o-minimal structures. 

In 1965, Feldman \cite{Feldman} proved that the set of maps, between two smooth manifolds, transverse to an $a$-regular smooth stratification in the target manifold is open in the strong topology. And, in 1979 Trotman \cite{Trotman} proved that $a$-regularity is actually necessary and sufficient for the openness of transverse maps. Thus, $a$-regular real stratifications can be characterized as those stratifications for which transverse maps form an open set in the strong topology. 

The result of Trotman does not hold in the complex setting, see Section \ref{Examples}. In this article we prove an analogue of Trotman's result for complex analytic stratifications.  

Let $M$ and $N$ be complex manifolds. Denote by $\al H(M,M)$ the set of all holomorphic maps between $M$ and $N$ with the weak topology.

Our first result is:

\begin{thm}\label{11} Let $M$ be a Stein manifold and $N$ be an Oka manifold\footnote{Recall that a complex manifold $N$ is said to be Oka if holomorphic maps from any Stein manifold to $N$ satisfy the $h$-principle with interpolation on compact sets in the source manifold}. Then, for any stratification $\Sigma$ of a complex analytic subvariety in $N$, the set of maps $\{f \in \al H(M,N) : f \pitchfork \Sigma\}$ is dense in $\al H(M,N)$.
\end{thm}

Our main result is:

\begin{thm}\label{12} Let $M$ be a Stein manifold and $N$ be an Oka manifold. Let $\Sigma$ be a stratification of a complex analytic subvariety in $N$. Let $r$ be the minimum of the dimensions of strata in $\Sigma$. If $\dim M = \dim N - r$ and there exists a compact set $K$ in $M$ such that the set of maps $T_K = \{f \in \al H(M,N) : f \pitchfork_K \Sigma\}$ is open in $\al H(M,N)$, then $\Sigma$ is an $a$-regular stratification.
\end{thm}

\section{\label{Examples} Examples}

We give examples in the complex case where Trotman's result fails to hold. 

{\bf 1.} Let $M$ be a compact complex manifold. A compact complex manifold is not Stein and any holomorphic map from $M$ to $\bb C^n$ is a constant. Let $V$ be a complex analytic subvariety in $\bb C^n$ and $\Sigma$ a stratification of $V$. Then, a holomorphic map $g_x : M \rightarrow \bb C^n$ which maps all points of $M$ to $x \in \bb C^n$ is transverse to $\Sigma$, at any point $m \in M$ if and only if $x \notin V$.  Thus, even if $\Sigma$ is not $a$-regular, the set of maps which are transverse to $\Sigma$ is open. 

{\bf 2.} A complex manifold $N$ is  `Brody hyperbolic' if there are no non-constant holomorphic maps from $\bb C$ to $N$, see \cite{Coskun} or \cite{Larusson}. On the other hand, complex manifolds which satisfy the Oka property are those manifolds which are the target of `many' non-constant holomorphic maps from $\bb C^n$. Brody hyperbolic manifolds do not satisfy the Oka property, see the discussion in \cite{Larusson}.  Let $M$ be the complex line and $N$ be a polydisk in $\bb C^n$. The manifold $N$ is a Brody hyperbolic manifold because any holomorphic map from $\bb C$ to $N$ must be constant by Liouville's theorem. So, once again a map from $\bb C$ to $N$ is transverse to a stratified set in $N$ if and only if it does not touch the stratified set and so even for non $a$-regular stratifications the set of maps transverse to them is open. Thus the result does not hold in this case. 

{\bf 3} Let $M = \bb C^2$ and $N$ be a complex $2$-hyperbolic manifold\footnote{A complex $k$-hyperbolic manifold $N$ is a manifold of dimension\footnote{All the dimensions considered here are complex dimension} $\geq k$ such that any holomorphic map from $C^k$ to $N$ has rank strictly less than $k$. For example, $\bb CP^3$ - 5 hyperplanes in general position is $2$-hyperbolic.} of dimension $4$. Let $\Sigma = \{X, Y\}$ be a stratification with two strata,  $X$ of dimension $2$ and $Y$ of dimension $3$. 

Suppose $Y$ is not $a$-regular over $X$. Then there exists a sequence $\{y_n\}$ in $Y$, a point $x \in X \cap \overline{Y}$ such that $y_n \rightarrow x$, $T_{y_n} Y \rightarrow \tau$ but $T_xX \notin \tau$. If Trotman's result were to hold in this case, it would be possible to find a sequence of maps $\{f_n\}$ from $M$ to $N$ which converges to $f$ in the weak topology, a point $w \in M$ such that $f_n(w) = y_n$ for large $n$ such that $f_n$ not transverse to $Y$ but $f(w) = x$ and $f$ transverse to $X$.

Since the dimension of $X$ is two, the maps from a two dimensional complex manifolds transverse to $X$ much have rank atleast $2$ to even intersect $X$. But since $N$ is $2$-hyperbolic, the only holomorphic maps from $M$ to $N$ have rank atmost $1$ and thus the only maps which are transverse are those whose image in $N$ does not intersect $X$. Thus Trotman's result does not hold in this case.

\section{Proof of Theorem \ref{11}} 

It is easy to see that $\al H(M,N)$ is a complete metric space (hence Baire). If we treat $M$ and $N$ as smooth manifolds, then the topology on $\al H(M,N)$ is the topology relative to the weak topology on the set of all smooth maps $C^{\infty}(M,N)$. Thus, by a result of Trivedi \cite{Trivedi} (also proved in Forstneri\v c \cite{Forstneric}), we immediately have:

\begin{lem}\label{31} Let $M$ and $N$ be complex manifolds and $\Sigma$ be an $a$-regular stratification of a complex analytic subvariety in $N$. Let $K \subset M$ be a compact subset of $M$. Then, the set
$T_K = \{f \in \al(H,M) : f \pitchfork_K \Sigma\}$
is open in $\al H(M,N)$.
\end{lem}

In the particular case when $\Sigma$ has only one stratum we get the complex version of the openness result of the Thom transversality theorem for the weak topology, i.e. the set of maps transverse to a closed submanifold on a compact subset of the source is open in the weak topology, see Trivedi \cite{Trivedi}. Moreover, if we treat $M$ and $N$ as smooth manifolds, then for any compact coordinate disk  $K$ in $N$ such that $K \cap \partial S = \emptyset$ , the set of maps transverse to $S\cap K$ in $N$ is open, and $S$ need not be closed. Thus again using the fact that the topology on $\al H(M,N)$ to the weak topology on the set of smooth maps $C^{\infty}(M,N)$, we immediately have:

\begin{lem}\label{32} Let $M$ and $N$ be complex manifolds and $S$ be a complex submanifold of $N$. Let $K$ be a compact coordinate disk in $N$ whose intersection with $\partial S$ is empty. Let $K'$ be a compact set in $M$. Then the set $T_{K'} = \{f \in \al H(M,N) : f \pitchfork_{K'} (S \cap K)\}$ is open in $\al H(M,N)$.
\end{lem}

Actually Forstneri\v c \cite{Forstneric} proved that the maps transverse to an $a$-regular stratification is dense. We show that $a$-regularity is not needed for the denseness. The reason why Forstneri\v c needs $a$-regularity is perhaps he did not observe the result of Lemma \ref{32}. We show how one can avoid $a$-regularity in the hypothesis. The idea is similar to that of Forstneri\v c but we will give a detailed proof because the ideas used will be applied later.

A complex manifold $N$ satisfies the \emph{Ellipticity condition}, $Ell_1$, introduced in `Partial Differential Relations' by Gromov \cite{Gromov, Gromov2}, if for every Stein manifold $M$, and every holomorphic map $f :M \rightarrow N$, there exists an integer $n \geq \dim N$ and a holomorphic map $F : M \times \bb C^n \rightarrow N$ such that the map $f = f_0 = F(\cdot,0) : M \rightarrow  N$ and $f^x=F(x,\cdot) : \bb C^n \rightarrow N$ is a submersion at $0 \in \bb C^n$ for each $x \in M$. 

\begin{lem}\label{33} Let $M$ be a Stein manifold and $N$ be an Oka manifold. Let $K$ be a compact subset of $M$ and $S \subset N$ a complex submanifold of $N$. Then the set $T_K = \{f \in \al H(M,N) : f \pitchfork_K S\}$ is dense in $\al H(M,N)$. 
\end{lem}

\begin{proof} Choose a holomorphic map $f \in \al H(M,N)$. Since $M$ is a Stein manifold and $N$ is an Oka manifold, $N$ satisfies the ellipticity condition $Ell_1$ (Proposition 4.6 in Forstneri\v c \cite{Forstneric}). Thus, there exists a map $F : M \times \bb C^n (n \geq \dim N)$ as in the definition of the condition $Ell_1$. Let $\pi : M \times \bb C^n \rightarrow C^n$ be the natural projection. Since $f^z : \bb C^n \rightarrow N$ is a submersion for all $z \in M$, there exist a small ball $D \subset C^n$ around the origin and an open set $U \subset M$ containing $K$ such that $F$ is a submersion from $V = U \times D$ to $N$. Then $S'=F^{-1}(S)\cap V$ is a complex submanifold of $V$. 

Now we show that if $t \in \bb C^n$ is a regular value of the restricted projection $\pi : S' \rightarrow D$ then $f_t \pitchfork_z$ if $(z,t) \in S'$.

If $(z,t) \in S'$ then $y \in f_t(z) \in S$. Since $F(z,t) = y$ and $F$ is a submersion, we know that
$$D_{z,t}F(T_{(z,t)}M \times \bb C^n) + T_yS = T_yN$$
that is, given any vector $a \in T_yN$, there is a vector $b \in T_{(z,t)}(M \times \bb C^n)$ such that 
$$D_{(z,t)}F(b) - a \in T_y S.$$
We want to exhibit a vector $v \in T_z M$ such that $D_zf_t(v)-a \in T_y S$. Now,
$$T_{(z,t)}M \times C^n = T_z M \times T_t\bb C^n,$$
so $b = (w,e)$ for vectors $w \in T_z M$ and $e \in T_t\bb C^n$. If $e=0$, we would be done, for since the restriction of $F$ to $M \times \{t\}$ is $f_t$, it follows that 
$$D_{(z,t)} F(w,0) = D_zf_t(w).$$
If $e \neq 0$, we may use the projection $\pi$ to kill it off. As
$$D_{(z,t)} : T_z M \times T_t\bb C^n \rightarrow T_t\bb C^n$$
is just the projection onto the second factor, and the fact that $t$ is a regular value of the restricted projection $\pi : S'\rightarrow D$, we know that there is some vector of the form $(u,e) \in T_{(z,t)} S$.  But, $F$ maps $S'$ to $S$, so $D_{(z,t)}F(u,e) \in T_y N$. Consequently, the vector $v = w -u \in T_zM$ is our solution, for
$$D_zf_t(v) - a = D_{(z,t)}F((w,e)-(u,e)) - a = (D_{(z,t)}F_{(w,e) -a }) - D_{(z,t)}F(u,e),$$
and both the latter vectors belong to $T_yN$.

By Sard's theorem \cite{Sard}, we see that the set of regular values of $\pi$ is dense in $D$. Choosing $t$ in this dense set and close to $0$ we get maps $f_t:M \rightarrow N$ transverse to $\Sigma$ on $K$ and which approximate $f$ on $K'$. 

Thus we have proved that the set $T_K = \{f \in \al H(M,N) : f \pitchfork_K S\}$ is dense in $\al H(M,N)$. 
\end{proof}

\begin{proof}[Proof of Theorem \ref{11}]

Cover $M$ by a countable collection of compact sets $\{K_{\alpha}\}$ in $M$ and cover each stratum $B_{\beta}$ in $\Sigma$ by a countable collection of compact coordinate disks $\{S_{\beta}^k\}$ of $N$ each of whose intersection with $\partial B_{\beta}$ is empty. By Lemma \ref{32} for every $\alpha$, $\beta$ and $k$, the set
$$T_{\alpha,\beta,i} = \{f \in \al H(M,N) : f \pitchfork_{K_{\alpha}} (B_{\beta} \cap K_i)\}$$  
is open in $\al H(M,N)$. And, by Lemma \ref{33} $T_{\alpha,\beta,i}$ is also dense in $\al H(M,N)$ because it contains the set $T_{\alpha,\beta}=\{f \in \al H(M,N) : f \pitchfork_{K_{\alpha}} B_{\beta}\}$. 

Notice that the set of maps $T=\{f \in \al H(M,N) : f \pitchfork \Sigma\}$ is actually the intersection of all sets of type $T_{\alpha,\beta,i}$. Thus, since $\al H(M,N)$ is a Baire space and the stratification is a locally finite stratification, $T$ is a countable intersection of open dense subsets in $\al H(M,N)$ and hence is itself dense. This proves Theorem \ref{11}.
\end{proof}

In fact a similar argument can be used to prove a more general result. Denote by $J^r(M,N)$ the set of jets of holomorphic maps between $M$ and $N$ and by $j^rf$ the $r$-jet extension of a holomorphic map $f: M \rightarrow N$. Combining the idea of Forstneri\v c and the idea above it is not difficult to see the following result.

\begin{thm} Let $M$ be a Stein manifold and $N$ be an Oka manifold. Let $\Sigma$ be a stratification of a complex analytic subvariety in $J^r(M,N)$. Then the set
of maps $\{f \in \al H(M,N) : j^rf \pitchfork \Sigma\}$ is dense in $\al H(M,N)$.
\end{thm}

\section{Proof of Theorem \ref{12}}

We need the following well known lemmas about openness of certain subsets of the set of all maps; we include a proof for the sake of completeness.

Treat $M$ and $N$ as $C^1$-manifolds.

\begin{lem} \label{immopen} Let $f \in C^1(M,N)$ be an immersion at some point $x \in M$. Then, there exist coordinate charts $(U,\phi)$ around $x$, $(V,\psi)$ around $f(x)$ with $f(U) \subset V$ and an $\epsilon > 0$, such that for all compact subsets $K \subset U$, every member of $\al N(f,(\phi,U),(\psi,V),K,\epsilon)$\footnote{Set of all maps $g \in C^1(M,N)$ such that $g(K) \subset V$, $||f_{\phi,\psi}(x) - g_{\phi,\psi}(x)|| < \epsilon$ and $||Df_{\phi,\psi}(x) - Dg_{\phi,\psi}(x)|| < \epsilon$ for all $x \in \phi(K)$}  is an immersion at each point of $K$.
\end{lem}

\begin{pf} Let $(\psi,V)$ be a chart at $f(x)$ and let $(\phi,U')$ be a chart at $x$ such that $f(U') \subset V$. 

Denote by $L(\bb R^m, \bb R^n)$, the set of all linear maps between $\bb R^m$ and $\bb R^n$; it is a normed space and has a well defined metric, say $\delta$. Let $\al{I}^c$ be the set of all non injective maps in $L(\bb R^m, \bb R^n)$, which is a closed subset. 

Put $\al D_{U'} = \{ D_{\phi(u)}f_{\phi,\psi}\in L(\bb R^m,\bb R^n): u \in U'\}$ and define $\eta : U'\rightarrow \bb R$ by $\eta(u) = \delta(D_{\phi(u)} f_{\phi,\psi},\al I^c)$. $\eta$ is a continuous map and since $f$ is an immersion at $x$, $\eta(x) >0$. Thus, there exists an open set $U \subset U'$ around $x$ such that $\eta(y)>0$ for all $y \in U$.

Now, for any compact set $K \subset U$, set $\epsilon = \min\{\eta(y) : y \in K\}$. We claim that the subbasic open neighbourhood of $f$, $\al N(f)= \al N(f,(\phi|_U,U),(\psi, V),K,\epsilon/2)$ has the required property. For, if $g \in \al N(f)$ and $y\in K$, 
\begin{eqnarray*}
 \delta(D_{\phi(y)}f_{\phi,\psi} , D_{\phi(y)} g_{\phi,\psi}) + \delta(D_{\phi(y)} g_{\phi,\psi}, \al I^c) &\geq& \delta(D_{\phi(y)}f_{\phi,\psi} ,\al I^c)\\
\delta(D_{\phi(y)} g_{\phi,\psi}, \al I^c) &\geq& \delta(D_{\phi(y)}f_{\phi,\psi} ,\al I^c)\\
&& - \delta(D_{\phi(y)}f_{\phi,\psi} , D_{\phi(y)} g_{\phi,\psi})\\
\delta(D_{\phi(y)} g_{\phi,\psi}, \al I^c) &\geq &\epsilon - \epsilon/2\\
\delta(D_{\phi(y)} g_{\phi,\psi}, \al I^c)& \geq& \epsilon/2 >0
\end{eqnarray*}
Thus, $g$ is an immersion at $y \in K$.
\end{pf}

\begin{lem} Let $Imm_K(M,N)$ be the set of maps between $M$ and $N$ which are immersions at each point of $K \subset M$. Then $Imm_K(M,N)$ is an open subset of $C^1_W(M,N)$ if $K$ is a compact set.
\end{lem}

\begin{pf}   
Let $f \in Imm_K(M,N)$. To prove that $ Imm_K(M,N)$ is open, we show that there exists an open neighbourhood of $f$ which is contained in $Imm_K(M,N)$. Since $f$ is an immersion at each $x \in K$, by lemma \ref{immopen} , for each $x \in K$ there exists a chart $U_x$ with the property that for each compact set $K_x \subset U_x$ there is a neighbourhood 
$N(f,(\phi_x, U_x),(\psi_x, V_x),K_x,\epsilon_x)$ such that each member of this neighbourhood is an immersion on all of $K_x$.  Since $K$ is compact, we can choose a finite subcollection $\{U_{x_1},\ldots,U_{x_r}\}$ of the coordinate neighbourhoods $\{U_x\}_{x \in K}$,  such that $K \subset \cup_{i=1}^r K_{x_i}$. But then the intersection 
$$\cap_{i=1}^r N(f,(\phi_{x_i}, U_{x_i}),(\psi_{x_i}, V_{x_i}),K_{x_i},\epsilon)$$ 
($\epsilon = \min\{\epsilon_{x_i}\}$)
is an open neighbourhood of $f$ and is contained in $Imm_K(M,N)$, as required.
\end{pf}

Now using the fact that the topology on $\al H(M,N)$ is the topology relative to the weak topology on $C^1(M,N)$, we immediately have:

\begin{lem} \label{43} Let $M$ and $N$ be complex manifolds and $Imm_K(M,N)$ be the set of all holomorphic maps between $M$ and $N$ which are immersions at each point of $K$. Then, $Imm_K(M,N)$ is open in $\al H(M,N)$ if $K$ is compact.
\end{lem}

Finally, we need a perturbation lemma which allows us to move holomorphic maps between Stein manifolds and Oka manifolds with certain properties. the idea of the proof is suggested by Franc Forstneric.

\begin{lem} \label{44} [Holomorphic perturbation lemma] Let $M$ be a Stein manifold and $N$ be an Oka manifold. Let $\{y_n\}$ be a sequence of point in $M$ converging to a point on $x$ in $N$. Let $f : M \rightarrow N$ be a holomorphic maps which is an immersion at a point $w \in M$ such that $f(w) = n$. Then there exists a sequence of maps $g_n : M \rightarrow N$ such that for large enough $n$, $g_n (w) = y_n$, $D_wf(T_wM)$ converges to $D_wg_n(T_wM)$ and $g_n$ converges to $f$ in the weak topology.
\end{lem}

\begin{proof} Choose a fat\footnote{A subset V of a topological space is called fat if $V \subset \text{cl}\,({\mathring V})$} $\al H(M)$-convex\footnote{Here $\al H(M)$ denotes the sheaf of holomorphic functions on $M$. Recall that an $\al H(M)$-convex subset of $M$ is a subset $U \subset M$ such that for every $K$ compact in $U$, $\hat K \subset U$ where $\hat K = \{x \in X : |f(x)| \leq \sup_K |f|  \,\,\forall f \in \al H(M)\}$.} set $V$ containing $x$ in the Stein manifold $M$. Consider the part $\{(m,f(m))\in M\times N : m \in V\}$ of the graph  of $f : M \rightarrow N$ in the product $M \times N$. This part of the graph can be covered by Stein neighbourhoods. In this collection of Stein neighbourhood we can certainly find maps with the desired property, i.e., there exists a holomorphic spray of maps $f_t : V \rightarrow N$, with $f = f_0$, depending holomorphically on the parameter $t$ in an open subset of $\bb C^n$ ($n = \dim N$). By varying $t$ we can find maps arbitrarily close to $f$ in the weak topology with the desired properties.

Now, since $N$ is an Oka manifold, this spray of maps over $V$ can be approximated by a spray of global maps over $M$ keeping $f = f_0$ fixed. Thus, giving global maps satisfying the given properties. 
\end{proof}

\begin{proof}[Proof of Theorem \ref{12}] Without loss of generality we can assume that $r \geq 1$. Let $N$ be of dimension $n$. Suppose that the stratification $\Sigma$ is not $a$-regular. Then there exists a pair $(X,Y)$ of strata in $\Sigma$, such that $X \cap \overline Y$ is non empty and $Y$ is not $a$-regular at a point $x \in X \cap \overline Y$. Thus, there exists a sequence $\{y_i\}$ in $Y$ such that $y_i$ tends to $x$ as $i \rightarrow \infty$, $T_{y_i}Y$ tends to $\tau$ as $i \rightarrow \infty$, but $T_xX \not\subset \tau$. Choose $v \in T_xX$ such that $v \not\in \tau$.

By a construction in Trotman \cite{Trotman}, existence of an  `$a$-fault', $x$, also implies that there exists a linear subspace $H$ of $T_x N$ not containing $v$ of dimension $n - r (= \dim M)$   such that 
\begin{equation}
\label{eq41}
H + T_xX  = T_x N
\end{equation}
\begin{equation}
\label{eq42}
H + \tau  \neq  T_x N
\end{equation}

We need to find a map $g : M \rightarrow N$ such that $D_wg(T_wM) = H$. Since $N$ is an Oka manifold, it satisfies the convex approximation property, which says that there exists a holomorphic map $\pi : \bb C^n \rightarrow N$ which is locally biholomorphic such that $f(0) = x$, see Forstneri\v c \cite{Forstneric2}. Choose a point $w \in K$, since $M$ is Stein it can be embedded in some $\bb C^p$ as a closed submanifold such that $w$ is mapped to $0$ under this embedding. 

Identify $T_x N$ by $\bb C^n$ and choose a basis of $T_x N$, $\{v_1,\ldots,v_l,\ldots,v_{n-1},v\}$ such that $\{v_1,\ldots,v_l\}$ span $\tau + H$ and $\{v_1,\ldots,v_{n-r}\}$ span $H$.

Choose a basis $\{u_1,\ldots,u_m,\ldots,u_p\}$ such that $\{u_1,\ldots,u_m\}$ span the vector subspace $T_wM$ of $\bb C^p$. And define a map $L : \bb C^p \rightarrow \bb C^n$ (this map is well defined because $p > n-r = m$) by 
$$L(a_1u_1 + \ldots + a_pu_p) = a_1 v_1 + \cdots + a_{n-r}v_{n-r}.$$
Now take $h = \pi \circ L : \bb C^p \rightarrow N$. The rank of $L$ is $n-r$ and since $\pi$ is a local biholomorphism, the rank of $Dh : \bb C^p \rightarrow T_xN$ is $n-r$ and moreover the image $D_0h(\bb C^p)$ is $H$. 

Now, let $h' : M \rightarrow N$ be obtained by restricting $h$ on $M$ via the embedding of $M$ in $\bb C^p$. Then, since the first $m$ vectors in $\bb C^p$, span the tangent space of $M$ at $w$, it is immediate that $D_wh'(T_wM) = H$ and since the dimension of $M$ is $n-r$, $h$ attains it's maximum rank at $w$ and so it is an immersion at $w$. In fact, the construction of shows that $h'$ is an immersion on whole of $M$. Thus, by equation \ref{eq41} $h'$ is transverse to $\Sigma$ at the point $w$ in $M$.

Thus we have the following situation:

$$
\Diagram
\bb C^p\rdTo^g & \rTo^L  & \bb C^n\\
\uInto^{i} & & \dTo^{\pi} \\
M & \rTo^h & N \\
\endDiagram
$$

Let $\fs D = \{f \in \al H(\bb M, N) : f(w) = x, D_wf(T_wM) = H\}$. This set is not empty since $h'$ belongs to $\fs D$. We show that there exists a map $h : M \rightarrow N$ in $\fs D$, which is transverse to $\Sigma$ on all of $M$. Denote $E_k = \{f \in \al H(M,N) : f|_K \,\, \text{is an immersion}\,\}$. By Lemma \ref{43}, $E_K$ is an open set of $\al H(M,N)$. 
Let $d$ be a metric on $\al H(M,N)$. Then, there exists a $\delta >0$ such that $$B_{\delta}(h') = \{f \in \al H(M,N) : d(h,f) < \delta\}$$ is a subset of $E_K$. Set $\fs E_K = \overline{B_{\delta/2}(h')} \cap \fs D$ (this set is not empty because $h' \in \fs E_K$) and moreover it is a closed subset of $\al H(M,N)$, since any converging sequence in $\fs E_K$ will converge to a point in $\fs E_K$ as all maps of the sequence are immersions on $K$. Since closed subsets of complete metric spaces are also complete metric spaces with the relative topology, we deduce that $\fs E_K$ is a Baire space. 

Using an argument similar to Theorem \ref{11}, we can show that the set $\{f \in \fs E_k : f \pitchfork \Sigma\}$ is dense in $\fs E_k$, which proves the existence of a map $h : M \rightarrow N$ which is transverse to $\Sigma$ on whole of $M$. 

Once we have $h$, we can now construct a sequence of maps $h_n : M \rightarrow N$, using Lemma \ref{44}, such that for large enough $n$, $h_n(w)=y_n$, $D_wh_n(T_wM)$ converges to $D_w h(T_wM) = H$. Moreover for large $n$ since $T_{y_n}Y$ converges to $\tau$, by \ref{eq42}, $h_n$ is not transverse to $Y$ at $w\in M$. This proves the existence of a sequence of holomorphic maps $h_n: M \rightarrow N$ not transverse to $\Sigma$ but whose limit is transverse to $\Sigma$, which implies that $T_K$ is not open giving a contradiction.
\end{proof}

Finally we remark that almost similar idea can be used to prove an analogue of our result in the case of algebraic manifolds and algebraic maps. We don't know if our result holds for a larger set of complex manifolds and it should be investigated.

\bibliographystyle{amsplain}
\bibliography{../mainbibliography}

\providecommand{\bysame}{\leavevmode\hbox to3em{\hrulefill}\thinspace}
\providecommand{\MR}{\relax\ifhmode\unskip\space\fi MR }
\providecommand{\MRhref}[2]{%
  \href{http://www.ams.org/mathscinet-getitem?mr=#1}{#2}
}
\providecommand{\href}[2]{#2}
\begin{thebibliography}{10}

\bibitem{Coskun}
I.~Coskun, \emph{The arithmetic and the geometry of {K}obayashi hyperbolicity},
  Snowbird lectures in algebraic geometry, Contemp. Math., vol. 388, Amer.
  Math. Soc., 2005, pp.~77--88.

\bibitem{Feldman}
E.~A. Feldman, \emph{The geometry of immersions. {I}}, Trans. Amer. Math. Soc.
  \textbf{120} (1965), 185--224.

\bibitem{Forstneric}
F.~Forstneri{\v{c}}, \emph{Holomorphic flexibility properties of complex
  manifolds}, Amer. J. Math. \textbf{128} (2006), no.~1, 239--270.

\bibitem{Forstneric2}
F.~Forstneri{\v{c}} and F.~L{\'a}russon, \emph{Survey of {O}ka theory}, New
  York J. Math. \textbf{17A} (2011), 11--38.

\bibitem{Gromov}
M.~Gromov, \emph{Partial differential relations}, Ergebnisse der Mathematik und
  ihrer Grenzgebiete (3) [Results in Mathematics and Related Areas (3)],
  vol.~9, Springer-Verlag, Berlin, 1986.

\bibitem{Gromov2}
\bysame, \emph{Oka's principle for holomorphic sections of elliptic bundles},
  J. Amer. Math. Soc. \textbf{2} (1989), no.~4, 851--897.

\bibitem{Larusson}
F.~L{\'a}russon, \emph{What is {$\ldots$} an {O}ka manifold?}, Notices Amer.
  Math. Soc. \textbf{57} (2010), no.~1, 50--52.

\bibitem{Mather3}
J.~N. Mather, \emph{Notes on topological stability}, Bull. Amer. Math. Soc.
  \textbf{49} (2012), 475--506.

\bibitem{Sard}
A.~Sard, \emph{The measure of the critical values of differentiable maps},
  Bull. Amer. Math. Soc. \textbf{48} (1942), 883--890.

\bibitem{Trivedi}
S.~Trivedi, \emph{Transversality theorems for the weak topology}, to appear in
  Proc. Amer. Math. Soc. (2012).

\bibitem{Trotman}
D.~J.~A. Trotman, \emph{Stability of transversality to a stratification implies
  {W}hitney {$(a)$}-regularity}, Invent. Math. \textbf{50} (1978/79), no.~3,
  273--277.

\bibitem{Whitney}
H.~Whitney, \emph{Tangents to an analytic variety}, Ann. of Math. (2)
  \textbf{81} (1965), 496--549.

\end{thebibliography}

\end{document}